\documentclass{amsart}
\usepackage{graphicx}
\usepackage{amsmath,amsthm,amsfonts,amssymb}
\newtheorem{thm}{Theorem}[section]

\newtheorem{lem}[thm]{Lemma}
\newtheorem{prop}[thm]{Proposition}
\theoremstyle{definition}

\theoremstyle{remark}

\numberwithin{equation}{section}
\newcommand{\set}[1]{\left\{#1\right\}}
\newcommand{\abs}[1]{\left\vert#1\right\vert}
\newcommand{\Q}{\mathbb Q}
\newcommand{\Z}{\mathbb Z}
\newcommand{\Norm}{\mathbf N}
\newcommand{\OK}{\mathcal O_K}
\newcommand{\C}{\mathbb C}
\newcommand{\R}{\mathbb R}
\DeclareMathOperator{\diam}{diam}
\renewcommand{\Re}{\mathrm{Re}}
\title[Discrete Quartic Integers]{Discrete Subsets of Totally Imaginary Quartic Algebraic Integers in the Complex Plane}
\author{Wenhan Wang}
\email{wangwh@math.washington.edu}
\address{{Department of Mathematics, Box 354350} \\
{University of Washington, Seattle, WA 98195-4350}}%
\date{\today}
\begin{document}
\begin{abstract}
Algebraic integers in totally imaginary quartic number fields are not discrete in the complex plane under a fixed embedding, which makes it impossible to visualize all integers in the plane, unlike the quadratic imaginary algebraic integers. In this note we consider a naturally occurring discrete subset of the algebraic integers with similar properties as lattices. For the fifth cyclotomic field, we investigate those integers with absolute values under a fixed embedding in a given bound. We show that such integers form a discrete set in the complex plane. It is observed that this subset has quasi-periodic appearance. In particular, we also show that the distance between a fixed point to the most adjacent point in this subset takes only two possible values.
\end{abstract}
\maketitle

\section{Introduction}

Let $K$ be a totally imaginary quartic number field, and denote by $\OK$ its ring of integers. $K$ has a unique maximal real subfield $K_0 = K\cap \R$. Let $\sigma$ denote the non-trivial automorphism of $K_0$, then $\sigma$ can be extended to two embeddings of $K\hookrightarrow\C$. By abuse of notation we denote a fixed one of the embeddings by $\sigma$ and the other embedding is then the complex conjugate of $\sigma$. For $z\in \OK\subseteq \C$, we denote $\sigma(z) = z^{\sigma}$ the image under the embedding $\sigma$.

Let $\mathcal B$ be a bounded subset of $\C$ containing $0$ as an interior point. Consider the set
$$\mathcal S_{\mathcal B} := \set{z\in\OK|z^{\sigma}\in\mathcal B}.$$
We claim that the set $\mathcal S_{\mathcal B}$ is a discrete subset of $\C$.

\begin{prop}
Suppose $z_1,z_2\in\mathcal S_{\mathcal B}$, then $\abs{z_1-z_2}\leq \frac{1}{2\diam(\mathcal B)}$.
\end{prop}
\begin{proof}
Since $z_1,z_2\in\mathcal S_{\mathcal B}$, we have
$$\Norm_{K/\Q}(z_1 - z_2) = \abs{z_1 - z_2}^2\abs{z_1^{\sigma} - z_2^{\sigma}}^2 \leq
4\diam(\mathcal B)^2\cdot\abs{z_1 - z_2}^2.$$
On the other hand, as $z_1$ and $z_2$ are algebraic integers, we have $\Norm_{K/\Q}(z_1 - z_2)\geq 1$. Hence we obtain
$$\abs{z_1 - z_2}\geq \frac{1}{2\diam(\mathcal B)}.$$
\end{proof}

\section{A Discrete Subset of Integers in $\Q(\zeta_5)$}

Consider the fifth cyclotomic field $K = \Q(\zeta_5)$ and its ring of integers $\OK$. Fix the embedding $\sigma$ of $K$ that sends $\zeta_5 =\exp(2\pi i/5)$ to $\zeta_5^2$. Let $\mathcal B$ be the unit circle in $\C$, and consider the discrete subset $\mathcal S = \mathcal S_{\mathcal B}$. We claim that $\mathcal S$ has the five-fold symmetry.

\begin{lem}
If $z\in \mathcal S$, then $\zeta z\in \mathcal S$.
\end{lem}
\begin{proof}
It suffices to observe that $\abs{\zeta} = 1$ and hence $\abs{\zeta z} = \abs{z}$ for $\zeta$ any root of unity.
\end{proof}

Note that $\mathcal S$ is a subset of $\OK$, and $S$ is in fact not a lattice in $\C$. However, $\mathcal S$ shares the following similar property as lattices in $\C$. The following figure depicts a portion of the set $\mathcal S\subseteq\C$ near the origin. The points surrounded by small circles are 0 and the fifth roots of unity.

\includegraphics[width = 5in]{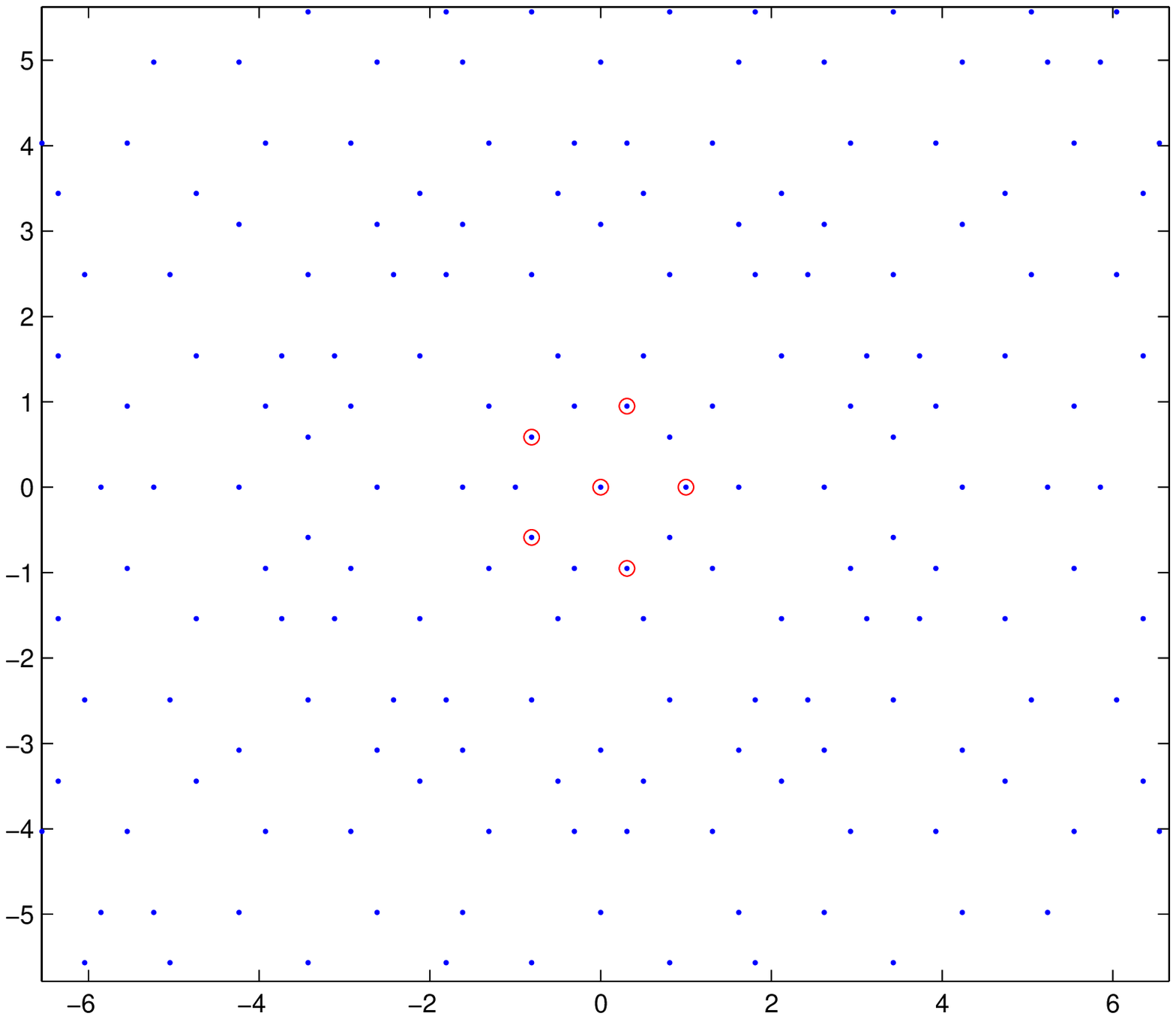}

For a fixed point $z$ in a given lattice $\Lambda\in\C$, the minimum distance from $z$ to another point $z'\in\Lambda$ is always a constant. Namely, if $\Lambda = \Z\omega_1 + \Z\omega_2$, then $\min_{z'\in\Lambda-\set{z}}\abs{z'-z} = \min_{z'\in\Lambda-\set{0}}\abs{z'}$. For the set $S$, we have the following theorem.

\begin{thm}\label{thm:minimum-distance}
$\min_{z'\in\mathcal S-\set{z}}\abs{z' - z} \in\set{\frac{\sqrt{5}-1}{2},1}$.
\end{thm}

To prove the above theorem, we need the following lemmas.

\begin{lem}\label{thm:unit}
Suppose $z_1,z_2\in S$. If $\abs{z_1 - z_2} < \frac{\sqrt{5}}{2}$, then $z_1 - z_2$ is a unit.
\end{lem}
\begin{proof}
Suppose to the contrary that $z_1 - z_2$ is not a unit, then $\Norm(z_1 - z_2) > 1$. Since $\OK$ is a PID, and both $2$ and $3$ are inert primes in $K$, we deduce that $\Norm(z_1 - z_2) \geq 5$. Hence $\abs{z_1 - z_2}\cdot\abs{z_1^{\sigma} - z_2^{\sigma}}\geq \sqrt{5}$. As $\abs{z_1^{\sigma} - z_2^{\sigma}}\leq 2$, we have $\abs{z_1 - z_2}\geq\frac{\sqrt{5}}{2}$. This contradicts with our assumption.
\end{proof}

\begin{prop}\label{thm:lower-bound}
Suppose $z_1,z_2\in S$. Then $\abs{z_1 - z_2}\geq\frac{\sqrt{5}-1}{2}$.
\end{prop}
\begin{proof}
With no loss of generality we assume $\abs{z_1 - z_2}<\frac{\sqrt{5}-1}{2}$. Then from the above lemma we know that $z_1 - z_2$ is a unit in $K$. Hence $\abs{z_1 - z_2} = (\frac{\sqrt{5}-1}{2})^j$ for some $j\geq 0$. We need to show that $j\leq 1$.

Suppose to the contrary that $j\geq 2$, then $\abs{z_1^{\sigma} - z_2^{\sigma}}\geq\frac{3+\sqrt{5}}{2}$. This gives rise to a contradiction as $\abs{z_1^{\sigma} - z_2^{\sigma}} \leq 2$.
\end{proof}

\begin{prop}\label{thm:upper-bound}
$\min_{z'\in\mathcal S-\set{z}}\abs{z - z'}\leq 1$.
\end{prop}
\begin{proof}
It suffices to show that at least one of $z + e^{\pi i j/5}$ is in $S$ for $0\leq j\leq 9$. Note that by definition of $S$ this is equivalent to $\abs{z^{\sigma} + \sigma(e^{\pi ij/5})}\leq 1$. Note that the inequality holds for $z = 0$ and all $0\leq j\leq 9$. Now assume that $z\neq 0$. Then we may choose $j'$ such that the argument of $z^{\sigma}$ and $e^{\pi ij'/5}$ differs by no greater than $\pi/10$. Thus
\begin{eqnarray*}
\abs{z^{\sigma} - \sigma(e^{\pi ij'/5})}^2 &=& \abs{z^{\sigma}}^2 + 1 - 2\Re z^{\sigma}e^{-\pi ij'/5} \\
&\leq& \abs{z^{\sigma}}^2 + 1 - 2\abs{z^{\sigma}}\cos(\pi/10) \\
&=& \left(\abs{z^{\sigma}} - \cos(\pi/10)\right)^2 + \sin^2(\pi/10) \\
&\leq& \cos^2(\pi/10) + \sin^2(\pi/10) = 1.
\end{eqnarray*}
\end{proof}

\begin{proof}[Proof of Theorem \ref{thm:minimum-distance}]
From Proposition \ref{thm:upper-bound} we know that the minimum distance is always less than or equal to 1. From Proposition \ref{thm:lower-bound} it follows that the minimum distance is greater than $\frac{\sqrt{5}-1}{2}$. By Lemma \ref{thm:unit}, the minimum distance is a real unit of $K$ in the interval $[\frac{\sqrt{5}-1}{2},1]$. Thus it takes values only in $\set{\frac{\sqrt{5}-1}{2},1}$.
\end{proof}

\

\end{document}